\documentclass[12pt]{amsart}
\usepackage{amsmath,amsthm,amsfonts,amscd,amssymb,eucal,latexsym,mathrsfs,enumerate}
\usepackage{hyperref}

\newcommand{\cl}[1]{{\mathcal{#1}}}

\newtheorem{theorem}{Theorem}[section]
\newtheorem{lemma}[theorem]{Lemma}
\newtheorem{prop}[theorem]{Proposition}
\theoremstyle{definition}
\newtheorem{definition}[theorem]{Definition}

\theoremstyle{remark}

\numberwithin{equation}{section}

\def\ca{{\mathcal A}}

\def\cc{{\mathcal C}}
\def\cd{{\mathcal D}}
\def\ce{{\mathcal E}}
\def\cf{{\mathcal F}}
\def\cg{{\mathcal G}}

\def\cl{{\mathcal L}}
\def\cam{{\mathcal M}}

\def\car{{\mathcal R}}
\def\cas{{\mathcal S}}
\def\ct{{\mathcal T}}

\def\bc{{\mathbb C}}

\def\bn{{\mathbb N}}

\def\br{{\mathbb R}}

\def\bt{{\mathbb T}}

\def\a{\alpha}

\def\d{\delta}

\numberwithin{equation}{section}

\def\ca{{\mathcal A}}

\def\cc{{\mathcal C}}
\def\cd{{\mathcal D}}
\def\ce{{\mathcal E}}
\def\cf{{\mathcal F}}
\def\cg{{\mathcal G}}

\def\cl{{\mathcal L}}
\def\cam{{\mathcal M}}

\def\car{{\mathcal R}}
\def\cas{{\mathcal S}}
\def\ct{{\mathcal T}}

\def\bc{{\mathbb C}}

\def\bn{{\mathbb N}}

\def\br{{\mathbb R}}

\def\bt{{\mathbb T}}

\def\a{\alpha}

\def\d{\delta}

\begin{document}
\title[Weak Derivatives and Reflexivity] {Higher Weak Derivatives and Reflexive Algebras of Operators }
\author{Erik Christensen}
\address{Dept. Math. U. Copenhagen, Denmark}
\curraddr{}
\email{echris@math.ku.dk}
\dedicatory{Dedicated to R. V. Kadison on the occasion of his ninetieth birth day.}
\date{\today}

\begin{abstract}
Let $D$ be a self-adjoint operator on a Hilbert space $H$ and $x$ a bounded operator on $H.$ We say that $x$ is $n$ times weakly $D-$differentiable, if for any pair of vectors $\xi, \eta$ from  $H$ the function  $ \langle e^{itD}xe^{-itD}\xi, \eta\rangle$ is $n$ times differentiable. We give several characterizations of $n$ times weak differentiability, among which, one is original. These results are used to show that for a von Neumann algebra $\cam$ on $H$ the algebra of $n$ times weakly $D-$differentiable operators in $\cam$ has a natural representation as a reflexive subalgebra of $B(H\otimes \bc^{(n+1)}).$  
\end{abstract}

\maketitle

\section{Introduction}

Let $D$ be a self-adjoint, {\em usually unbounded,} operator on a Hilbert space $H$ and $x$ a bounded operator on $H,$  then {\em Quantum Mechanics,} \cite{JoN} {\em Operator Algebra} \cite{KR} and {\em Noncommutative Geometry} \cite{AC} offer plenty of reasons why we should be interested in operators that are formed as commutators $[D,x] = Dx - xD.$ In  noncommutative geometry we want to find a set-up such that classical smooth structures may be described in a language based on operators on a Hilbert space. A derivative is described in terms of a commutator $[D,x]$ and a higher derivative via an iterated commutator $[D, [D, \dots,[D,x]\dots ]],$ so a basic question is to determine the set of operators for which such an iterated commutator makes sense.  It is not clear when a commutator such as $[D,x]$ is densely defined and bounded on its domain of definition, and for two bounded operators $x, y $ such that $[D,x]$ and $[D,y]$ are bounded and densely defined the sum of the commutators and/or  the commutator 
$[D,xy]$  may not be densely defined, so the expression
 $[D,x]$ does not define a derivation on a subalgebra of $B(H)$ in a canonical way.  In the article \cite{EC} we realized that the concept we named weak $D-$differentiability provides a set-up, which may be used to decide for which bounded operators $x$ the commutator $[D,x] $ should be defined. We say that a bounded operator $x$ on $H$ is weakly $D-$differentiable if for each pair of vectors $\xi, \eta$ in $H$ the function $ \langle e^{itD}xe^{-itD}\xi, \eta\rangle$ is differentiable. For a weakly $D-$differentiable operator  $x$ the commutator $[D,x]$ is then defined and bounded on all of the domain of $D,$ so it is possible to define a derivation $\d_w$ from the algebra of weakly $D-$differentiable operators into $B(H).$   We were later informed that the concept of weak $D-$differentiability, the algebra property of the weakly differentiable operators and the derivation $\d_w$ are well  known by researchers in mathematical physics \cite{AMG} and \cite{GGM}, so according to the notation of the  book \cite{AMG}, see page 192,  the mentioned algebra is $C^1(D,H).$ We will adopt this notation but modify it such that it makes it possible to look at those elements of a C*-algebra $\ca$ acting on $H$ which are weakly $D-$differentiable. This subalgebra of $\ca$ is then denoted $C^1(\ca, D).$   First of all we will like to study the algebra of higher weak derivatives, which with the notation of \cite{AMG} is the algebra $C^n(D,H),$ and the algebra of $n$ times weakly $D-$differentiable operators inside a C*-algebra $\ca$ on $H$ is $C^n(\ca, D) := C^n(D,H) \cap \ca.$
 In section 4 we give several characterizations of those operators that are $n$ times weakly $D-$differentiable, and we would like to mention here, that a bounded operator $x$ is $n$ times
 weakly $D-$differentiable if and only if for any $k $ in $\{1, \dots, n\} $ the $k'$th commutator $[D,[ D, \dots,[D,x]\dots]] $ is defined and bounded on dom$(D^k).$ This is known to many mathematicians, but we could not find a reference where the details are easy to follow, so we have included a proof here. This characterization of $n$ times weak $D-$differentiability  will be crucial for the results of section 5 on reflexive algebras. We also give a characterization of higher weak differentiability based on an embedding of the higher commutators $[D, [D, \dots [D, a ] \dots ]] $ into a linear space consisting of infinite matrices of  bounded operators.  This set-up is original, and we hope that it will turn out to be a useful frame inside which some operator theoretical  questions can be dealt with in a way which avoids the tiresome considerations of the validity of products and sums of operators. After the article \cite{EC} was accepted for publication and proof read, we realized, that the one parameter group of automorphisms of $B(H)$ given by $B(H) \ni x \to e^{itD}xe^{-itD} \in B(H)$  is actually a so-called {\em adjoint semigroup  on a dual Banach space}.  Adjoint semigroups were first studied in \cite{RP}, and \cite{JaN} contains a survey of the general theory of adjoint semigroups. Our usage of the general theory is  limited, but several things could have been presented in an easier way in \cite{EC}, if we had been able to make references to \cite{JaN}.

\section{  Weak and higher  weak differentiability} 
In order to avoid confusion  we will like to clear up a point which has not been presented in an optimal way in \cite{EC}. The Definition 1.1 of  \cite{EC} defines a bounded operator $x$ to be weakly $D-$differentiable if there exists a bounded operator $b$ on $H$ such that for any pair of vectors $\xi, \eta$ in $H$ we have
$$\underset{t \to 0}{\lim}|\langle \big( \frac{e^{itD}x e^{-itD} -x}{t} - b\big) \xi, \eta \rangle = 0.$$   
This definition implies that for any $\xi, \eta $ the function 
$$t \to \langle e^{itD}xe^{-itD}\xi, \eta \rangle$$
 is differentiable at $t=0,$ and it is stated, but not explicitly proven that this latter property implies weak $D-$differentiability as defined via a weak derivative $b.$ It is quite easy to see that the two sorts of weak $D-$differentiability are equivalent and all the arguments are presented in \cite{EC}, but the consequences are not made sufficiently clear. The right formal definition of weak $D-$differentiability then becomes as follows.
 \begin{definition}
A bounded operator $x$ on $H$ is weakly $D-$differenti-able if for any pair of vectors $\xi, \eta$ in $H$ the function $t \to \langle e^{itD} xe^{-itD} \xi, \eta \rangle $ is differentiable at $t=0.$ 
\end{definition} 

To see that our present definition of weak $D-$differentiability implies the existence of a weak derivative, i.e. a bounded operator $b$ such that Definition 1.1 of  \cite{EC} is satisfied, we refer the reader to the proof of (ii) $\Rightarrow $ (iii) in Theorem 3.8 of \cite{EC}. That step is the crucial part of the proof, and it is  based on the uniform boundedness principle applied to all the operators $$\{\frac{e^{itD}xe^{-itD} -x }{t} : t \neq 0 \}.$$ 
This set is bounded because any function such as $t \to  \langle e^{itD}xe^{-itD} \xi, \eta \rangle  $  is differentiable at $t=0,$ and hence the set of values  
$$\{\langle \frac{e^{itD}xe^{-itD} -x }{t} \xi, \eta \rangle : t \neq 0 \}$$
 is bounded and the principle applies.     
 The existence of $b$ then follows from the rest of Theorem 3.8 of \cite{EC}. We will quote that theorem below and define the higher weak derivatives, but first we will recall a couple of other forms of $D-$differentiability.
    
    We say that a bounded operator $x$ 
 is uniformly $D-$differentiable if the function $t \to e^{itD}xe^{-itD} $ is differentiable at $t=0,$ with respect to the norm topology on $B(H).$ In analogy with the definition of weak $D-$differentiability we say that $x$ is strongly $D-$differentiable if for each vector $\xi$ in $H$ the function $t \to e^{itD}xe^{-itD}\xi $  is differentiable  at $t =0$ with respect to the norm topology on $H.$  It follows from \cite{EC} that weak and strong $D-$differentiability are equivalent but uniform $D-$differentiability is in general    a stronger property.
 
The book \cite{AMG} studies strong $D-$differentiability in its Chapter 5, and it mentions that this concept is equivalent to weak $D-$differentiability, which we prefer to work with, because it seems to be closer to the classical concepts involving differentiable functions on $\br.$ Anyway we already have  adopted the notation from \cite{AMG}, but modified it so that the C*-algebra $\ca$ is part of the notation too, so we define:
 
\begin{definition} 
Let $\ca$ be a C*-algebra on a Hilbert space $H$ and $D$ a self-adjoint operator on $H.$ Then the algebra of $n$ times weakly $D-$differentiable operators in  $\ca$ is denoted $C^n(\ca, D)$. 
\end{definition}    
 
 The self-adjoint operator $D$ defines a one parameter automorphism group $\a_t$ on $B(H),$ which for a bounded operator $x$ on $H$ is defined by $\a_t(x) := e^{itD}xe^{-itD}. $ For a weakly $D-$differentiable operator $x$ in $B(H)$ it then follows, that $\d_w(x)$ is  the weak operator derivative  $\frac{d}{dt} \a_t(x)|_{t=0}, $  but  
 there is also the possibility of having a norm derivative of $\a_t(x)$ at 0, and in that case we let $\d_u(x)$ denote that derivative. 
  On the other hand, when speaking of higher derivatives, we quote  
 from \cite{EC} the following result, which tells that higher uniform derivatives are closely related to weak derivatives.
 \begin{theorem}
 Let $x$ be a bounded operator on $H$ and $n \geq 2.$ If $x$ is $n$ times weakly $D-$differentiable then $x$ is $n-1$ times uniformly $D-$differentiable. 
 \end{theorem}
 \begin{proof}
 See Corollary 4.2 of \cite{EC}. 
 \end{proof}
We will quote  Theorem 3.8 from \cite{EC} here,  without description of all the language used. Not all of the results below may be generalized to higher derivatives and for those properties, which can be extended, we will give the necessary precise definitions, when needed. Many of the results may be found in Section 2 of \cite{GGM}.

 \begin{theorem} \label{0}
Let $x$ be a bounded operator on $H.$ The following properties are equivalent:
\begin{itemize}
\item[(i)] $x$ is strongly $D-$differentiable.
\item[(ii)] $x$ is weakly $D-$differentiable.
\item[(iii)] $x$ is $D-$Lipschitz continuous.
\item[(iv)] The sesquilinear form $S(i[D,x])$ on the domain of $D$ is bounded.
\item[(v)]  The infinite matrix $m(i[D, x]) $ represents a bounded operator.
\item[(vi)] The operator $Dx - xD$ is defined and bounded on a core for $D.$
\item[(vii)] The operator $Dx - xD$ is bounded and its domain of definition is $\mathrm{dom}(D).$ 
\end{itemize}
If $x$ is weakly $D-$differentiable then  
\begin{align*} \forall \xi, \eta \in H \quad  \underset{t \to 0}{\lim} \frac{\langle (e^{itD}xe^{-itD} - x)\xi, \eta\rangle } {t} = & \langle \d_w(x) \xi, \eta\rangle  \\ x\,\mathrm{dom}D \subseteq \mathrm{dom}D  \text{ and }  \d_w(x)\big| \mathrm{dom}(D) = & i(Dx - xD) \\
\forall t \in \br: \quad \|\a_t(x) - x\| \leq & \|\d_w(x)\| |t|.
\end{align*}
\end{theorem}

The properties (iii) and (iv) from the theorem just above have no simple generalizations to higher derivatives and will not be discussed here at all. The remaining five properties all suggest natural  extensions to the setting of higher  weak derivatives and higher commutators as well, and we will discuss this in the next section. 

Before embarking into the study of  higher weak derivatives we would like to make the following observation explicit. The reason being, that although most people know it, we do not have an exact reference at hand.  
\begin{lemma}
\label{der}
 If a bounded operator $x$ on $H$ is weakly $D-$differentiable then for any pair of vectors $\xi, \eta$ in $H$ 
the function $\langle\a_t(x) \xi, \eta\rangle $ is differentiable on $\br$ and $$\frac{d}{dt}\langle \a_t(x) \xi, \eta\rangle = \langle \a_t(\d_w(x)) \xi, \eta\rangle.$$
\end{lemma}

\begin{proof} 
By definition the equality holds for 
$t=0,$ and arguments similar to the ones given in the proof of Lemma 2.1 of \cite{EC} show that the identity may be translated from $t =0$ to any other real $t.$
\end{proof}

This lemma has an immediate consequence, which we formulate as a proposition, since it is important, although its proof is trivial.

\begin{prop} \label{highder}
A bounded operator $x$ on $H$ is $n$ times weakly $D-$dif-ferentiable if and only if  
$x$ is in $\mathrm{dom}(\d_w^n)$ and if and only if for any pair $\xi, \eta $ in $H$ the function $\langle\a_t(x)\xi, \eta\rangle $ is in $C^n(\br).$ 

If $ x$ is $n$ times weakly differentiable then 
$$ \frac{d^n}{dt^n} \langle \a_t(x)\xi, \eta\rangle = \langle \a_t( \d_w^n(x) )\xi, \eta\rangle.$$
\end{prop}
\begin{proof}
Follows from Lemma \ref{der} by induction.
\end{proof}

We will end this section by introducing a norm on $C^n(D,H).$
 \begin{definition} For any $x $ in $C^n(D,H)$ the norm $\|x\|_n $ is defined by
 $$\|x\|_n = \sum_{j=0}^n \frac{1}{j!}\|\d_w^j(x)\|.$$
 \end{definition}
 
 This is not the same norm as the one defined in \cite{AMG} Definition 5.1.1 at page 195, but it is equivalent to that norm,  and it follows from  \cite{EC} Proposition 3.10 that $\big( C^n(\ca, D), \|.\|_n\,\big)$ is a Banach algebra. 
   
\section{ Higher weak derivatives and iterated commutators.}
Having Proposition \ref{highder} one might think that our understanding of $\d_w$ and its powers is sufficiently well established for most purposes, but it is not. The problem is that we do not know how to relate higher weak derivatives to expressions involving iterated commutators with $iD$. If $x$ is in dom$(\d_w)$ then it follows from Theorem \ref{0} that $iDx - ixD$ is defined on all of  dom$(D)$ and   $\d_w(x)$ is the closure of  $iDx - ixD. $ If $x$ is in dom$(\d_w^2),$ then it is natural  to look at the second $iD$  commutator  $$iD(iDx - ixD) -   (iDx - ixD)(iD),$$ but we know nothing about its domain of definition, possible boundedness and closure. In this section we will show that the properties of the higher commutators are as nice as we can possibly hope for. 
We will show that for a bounded   $n$ times weakly differentiable operator $x,$ the  $n$ times iterated commutator between $iD$ and $x$ is defined on dom$(D^n)$ and the closure of this operator equals $\d_w^n(x).$ We will base the proof of this on the results of  Theorem \ref{0}. In order to simplify the writings below we define an operator $d$ on the space of linear operators on $H.$

\begin{definition}
\begin{itemize}
\item[]
\item[(i)]
A linear operator on $ H$ is a linear operator defined on a subspace of $H$ and with values in $H.$ The space of all linear operators on $H$ is denoted $\cl.$ A product $yz$ of operators in $\cl$ is defined on those vectors in the domain of $z$ which are mapped into the domain of $y$ by $z,$ and a sum is defined on the intersection of the domains of all the summands.  
\item[(ii)] The operator $d$ on $ \cl$ is defined for $y$ in $\cl$ by $d(y) := (iD)y - y(iD).$ 
\end{itemize}
\end{definition}
 We will start our investigation on higher commutators by  making the following observation. 

 \begin{lemma} \label{dn}
 Let $x$ be a bounded operator in $B(H)$ and $n$ a natural number. If $x$ is $n$ times weakly differentiable then  for any $k$ in $\{1, \dots , n\}:$ \begin{align*}
  & \d_w^{k-1}(x) : \mathrm{dom}(D) \to \mathrm{dom}(D), \\
  &\d_w^k (x) | \mathrm{dom}(D) = i[D, \d_w^{k-1}(x)] = d(\d_w^{k-1}(x))
\end{align*}
 \end{lemma}

\begin{proof}
If $x$ is $n$ times weakly differentiable, then  for any $k$ in $\{1, \dots , n\} $ we have $\d_w^{k-1}(x)$ is in dom$(\d_w).$ Then Theorem \ref{0} item (vii) presents the claimed properties of $\d_w^{k-1}(x).$ 

\end{proof}

The  statements in Lemma \ref{dn} show that $\d_w^k(x)$ is the closure of the commutator $[iD, \d_w^{k-1}(x)], $ but if $k>1$ then $\d_w^{k-1}(x) $ is defined as a closure of the commutator $[iD, \d_w^{k-2}(x)], $ so we have no direct control over the operator $[iD, \d_w^{k-1}(x)]. $ This is not sufficient for our purpose, so we want to look at the restriction of  such a commutator to dom$(D^k),$ and then show that on this domain the higher weak derivative   may be computed without any closure operations, as a higher commutator, and that the closure of this algebraically defined commutator equals $\d_w^k(x).$ 
\begin{prop} \label{dcom}
Let $x$ be an $n$ times weakly differentiable bounded operator on $H,$ then for $k$ in $\{1, \dots, n\}$ 
$$\begin{matrix}
(i) &\d_w^{k-1}(x)\mathrm{dom}(D) &\subseteq &\mathrm{dom}(D)\\
(ii) & x\, \mathrm{dom}(D^k) &\subseteq &\mathrm{dom}(D^k), \\
(iii) & \mathrm{dom}(d^k(x)) & = &\mathrm{dom}(D^k) \\  
(iv) &  \d_w^{k}(x)|\mathrm{dom}(D^k) & = & d^k(x) \\ 
(v) & \d_w^{k}(x) & =  & \mathrm{closure}(d^k(x)) . 
\end{matrix}$$ 
\end{prop} 

\begin{proof}
The item (i) follows from Lemma  \ref{dn}. The following four items are related and we show them by induction on $k.$ For $k = 1 $ the results follow again from item (iv) of Theorem \ref{0}. Then suppose $ 1 < k \leq n$ and that the statements are true for natural numbers in the set $\{1, \dots, k-1\}. $ We start by proving (iii), so we will choose a vector $\xi$ in dom$(D^k),$
then $\xi $ is in dom$(D^{k-1})$ so $d^{k-1}(x)\xi = \d_w^{k-1}(x)\xi $  and by item (i) $d^{k-1}(x)\xi $ is in dom$(D),$ and finally $\xi$ is in dom$(iDd^{k-1}(x)).$ By assumptions $(iD)\xi$ is in dom$(D^{k-1})$ which equals dom$(d^{k-1}(x)) $ so $\xi $ is in  dom$(d^{k-1}(x)(iD))$ too, and dom$(D^k) \subseteq \mathrm{dom}(d^k(x)).$ The opposite inclusion is trivially true since  $d^k(x)$ is a sum of terms, where the last summand is $(-i)^kxD^k.$  

With respect to item (iv), note  that
\begin{displaymath} D\mathrm{dom}(D^k) \subseteq  \mathrm{dom}(D^{k-1}) \subseteq  \mathrm{dom}(D ),
\end{displaymath}
so by the induction hypotheses the domain  for 
$d^{k-1}(x)D$ equals dom$(D^k)$ and $d^{k-1}(x)D = \d^{k-1}_wD\big|\mathrm{dom}(D^k). $  By (i)  and the induction hypotheses   $D\d_w^{k-1}(x)$ is defined on dom$(D^k)$ and equals  $Dd^{k-1}(x)$ on that domain. Hence item  (iv) follows. 

With respect to (v) we remark, that dom$(D^k)$ is a core for $D$ since it contains the vectors in the core $\ce,$ which was introduced in the proof of (v) $ \Rightarrow$ (vi) in Theorem 3.8 of \cite{EC}. Then $\d_w^k(x) $ is the closure of the commutator $d(\d_w^{k-1}(x))| \mathrm{dom}(D^k),$ but the latter equals $d^k(x)$ so (v) follows.

To prove (ii) we remark, that from (i) and (iv) it follows that 
$$d^{k-1}(x) \mathrm{dom}(D^k) \subseteq \mathrm{dom}(D).$$
On the other hand a closer examination of the expression $d^{k-1}(x)\xi$ for a vector $\xi $ in dom$(D^k)$ shows that 
\begin{equation}
d^{k-1}(x)\xi = (i)^{k-1} \sum_{j=0}^{k-1} \binom{k-1}{j}(-1)^jD^{k-1-j}xD^j\xi \label{Dk} 
\end{equation}
For $j > 0 $ we have $D^j\xi$ is in dom$(D^{k-j})$ and by assumption $xD^j\xi$ is in dom$(D^{k-j})$ so $D^{k-1-j}xD^j\xi$ is in dom$(D).$ Then for $j = 0$ we find that $D^{k-1}x \xi$  may be written as a difference of two vectors in dom$(D)$ and hence $x\xi$ is a vector in dom$(D^k),$ and item (ii) is  proven.  
\end{proof}

\section{Equivalent Properties}

In analogy with the results of Theorem \ref{0} we want to show that higher order weak differentiability may be characterized in several different ways. Some of the properties we find are expressed in terms of infinite matrices of operators, so  we will include a  short description of this set-up here.  

In \cite{EC}  we defined a sequence of pairwise orthogonal projections with sum $I$ in $B(H)$ by letting $e_n$ denote the spectral projection for $D$ corresponding to the interval $]n-1, n].$ Then we defined $\cam$ to be all matrices $(y_{rc})$ with $r$  and $c$  integers and $y_{rc}$ an operator in $e_rB(H)e_c. $  Any bounded operator $x$ on $H$ induces an element $m(x)$ in $\cam$ which is defined as $m(x)_{rc} := e_rxe_c.$
 The operator $D$ has a representation $m(D)$  in $\cam$ too, and it is defined as a diagonal matrix $m(D)_{rc} = 0, $ if $r \neq c$ and diagonal elements $d_r := m(D)_{rr}  := De_r.$ Then for any element  $ y = (y_{rc})$ in $\cam,$ the commutator $i[m(D),y]$  makes sense in $\cam$  by $$ i[m(D),y]_{rc} := i( d_r y_{rc}  - y_{rc}d_c ),$$ and we may define a linear mapping $d_{\cam}: \cam \to \cam$ by $$\forall y =(y_{rc}) \in \cam: \, \, d_{\cam}(y)_{rc} := i d_ry_{rc} - i y_{rc}d_c.$$
By the computations above we get that the powers $d_{\cam}^n$ are given as  
\begin{equation} \label{mdn} 
\forall n \in  \bn \, \forall y = (y_{rc}) \in \cam: \quad  
d_{\cam}^n(y)_{rc} = i^n\sum_{k=0}^n \binom{n}{k}(-1)^{n-k} d  _r^k y_{rc}d_c^{n-k}.
\end{equation}

We can now formulate our result on characterizations of higher weak differentiability.

 \begin{theorem} \label{I}
Let $x$ be a bounded operator on $H$ and $n$ a natural number.
The following properties are equivalent:
\begin{itemize}
\item[(i)] $x$ is in $\mathrm{dom}(\d_w^n).$ 
\item[(ii)] $x$ is $n$ times  weakly $D-$differentiable.
\item[(iii)] $x$ is $n$ times strongly $D-$differentiable.
\item[(iv)] $\forall k \in \{1, \dots, n\}$ \begin{align*}
& x: \mathrm{dom}(D^k) \to \mathrm{dom}(D^k) \\
& d^k (x) \text{ is defined and bounded   on } \mathrm{dom}(D^k) \text{ with closure } \d_w^k(x).
\end{align*} 
\item[(v)]  For $k $ in $\{1, \dots, n\}$ the infinite matrix $d_{\cam}^k(m(x))) $ represents a bounded operator.
\item[(vi)] There exists a core $\cf$ for $D$ such that for any $k$ in $\{1, \dots , n\}$  the operator $d^k(x)$ is defined and bounded on $\cf.$
\end{itemize}
\end{theorem}  
\begin{proof}
We prove (i) $\Leftrightarrow$ (ii), (ii) $\Leftrightarrow$ (iii), (ii) $\Rightarrow$ (iv) $\Rightarrow$ (v) $\Rightarrow$ (ii)  and (ii) $\Leftrightarrow$ (vi).

(i) $\Leftrightarrow$ (ii):

\noindent
Follows from  Proposition \ref{highder}.

(ii) $\Rightarrow$ (iii):

\noindent
 Follows by an induction based on the following induction step. Suppose $ 0 \leq k < n , $ $x$ is a bounded  $n$ times weakly differentiable operator, which is  $k$ times strongly differentiable, then $\d_w^k(x)$ is the $k'$th strong derivative by Theorem \ref{0}, and since this operator is weakly differentiable, the same theorem shows that $\d_w^k(x)$ is strongly differentiable with strong derivative $\d_w^{k+1}(x).$

(iii) $\Rightarrow$ (ii):

\noindent
Follows from the Cauchy-Schwarz inequality.

(ii) $\Rightarrow$ (iv):

\noindent
This follows from Proposition \ref{dcom} 

(iv) $\Rightarrow$ (v):

\noindent
Let $ 1 \leq k \leq n,$ then we are given that $\d_w^k(x)$ exists and is a bounded operator such that $\d_w^k(x)| \mathrm{dom}(D^k) = d^k(x).$ Let $c$ be an integer then $e_cH \subseteq \mathrm{dom}(D^k)$  so for any integer $r$ we get $$e_r\d_w^k(x)e_c = e_rd^k(x)e_c = i^k \sum_{j=0}^k \binom{k}{j} (-1)^{k-j}e_rD^jxD^je_c = d_{\cam}^k(m(x))_{rc},$$ hence $d_{\cam}^k(x)$ is the matrix of a bounded operator and (v) follows. 

(v) $\Rightarrow$ (ii):

\noindent
Assume (v), i.e. that for any $k$ in $\{1, \dots, n\} $ there exists a bounded operator $z_k$ on $H$ such that for any pair of integers $r,c$ we have 
$e_rz_ke_c = d^k_{\cam}(x)_{rc}.$ The case  $k=1$ is covered by Theorem \ref{0}. The proof may be found in \cite{EC}, but we recall the main step, because we will use it repeatedly below. For any vector $\xi$ from $e_cH$ we showed  that $x\xi$ is in dom$(D).$ It then follows that for any  integer $r$ and a vector $\xi$ in $e_cH$ we have  $x\xi $ is in dom$(D)$ and $$ e_rz_1\xi = i(d_re_rxe_c - e_rxe_cd_c)\xi = ie_r(Dx -xD)\xi.$$
and we concluded that $x $ is weakly differentiable with $\d_w(x) = z_1.$ We may now assume that $1 < k \leq n$ and $x$ is weakly differentiable of order $k-1$ with $\d_w^j(x) = z_j$ for $1 \leq j \leq k-1.$ Then for $\xi $ in $e_cH$ we  get $\d_w^{k-1}(x)\xi$ is in dom$(D)$ so we have  
\begin{align*} e_rz_k\xi = & i(d_re_rd^{k-1}_{\cam}(x)e_c - e_rd^{k-1}_{\cam} (x)e_cd_c)\xi \\ = & ie_r(D\d_w^{k-1}(x) -\d_w^{k-1} (x)D)\xi. \end{align*} Hence $\d_w^{k-1}(x)$ is weakly differentiable and $\d_w^k(x) = z_k,$ so $x$ is $n$ times weakly differentiable and (ii) follows.

(ii) $\Rightarrow$ (vi):

\noindent
For any $n$ in $\bn,$  the space dom$(D^n)$ is a core for $D,$ so (vi) follows from (iv), which, in turn,  follows from (ii).
  
(vi) $\Rightarrow$ (ii):

\noindent
Now suppose (vi) holds for a bounded operator $x$ on $H.$ 
Then for $k$ in $\{1, \dots, n\}$ there exist bounded operators $y_k = \mathrm{closure}( d^k(x) | \cf).$ 
Let us look at the case $k=1$ first. Then $(iD)x - x(iD)$ is defined and bounded on the core $\cf$ for $D,$ so by Theorem \ref{0} item (vi) $x$ is in dom$(\d_w)$ and $y_1 = \d_w(x)$. Let us now suppose that $1 < k \leq n $ and we know that $y_{j} = \d_w^{j}(x),$ for $1 \leq j \leq k-1,$ then for any $\xi$ in $\cf$ we can find a sequence of vectors $\xi_n$ in $\cf$ such that $\xi_n \to \xi$ and $D\xi_n \to D\xi$ for $n \to \infty.$  Since $d^k(x)$ is bounded and defined on $\cf $ we have 
\begin{align*} 
y_k\xi = \underset{n \to \infty}{\lim}d^k(x) \xi_n = &\underset{n \to \infty}{\lim}\big((iD)d^{k-1}(x) \xi_n - d^{k-1}(x)(iD) \xi_n \big)\\
=& \underset{n \to \infty}{\lim}\big((iD)\d_w^{k-1}(x) \xi_n - \d_w^{k-1}(x)(iD) \xi_n \big).
\end{align*} Since the last part of these  equations forms a convergent sequence we find that $\underset{n \to \infty}{\lim}(iD)\d_w^{k-1}(x) \xi_n $ exists and 

$$ \underset{n \to \infty}{\lim}(iD)\d_w^{k-1}(x) \xi_n = y_k \xi + \d_w^{k-1}(x)(iD)\xi.$$
Hence $\d_w^{k-1}(x)\xi$ is in dom$(D)$ 
and $$y_k\xi = (iD)\d_w^{k-1}(x)\xi - \d_w^{k-1}(x)(iD)\xi$$ By Theorem \ref{0} we get that $\d_w^{k-1} (x)$ is weakly differentiable and $\d_w^k(x) = y_k,$ so $x$ is $n$ times weakly differentiable, and the theorem follows.
\end{proof}

\section{Reflexive representations of the algebras of higher weakly differentiable operators in a von Neumann algebra.}

In this section we will consider the case where we are dealing with a von Neumann algebra $\cam$ on $ H$ and study aspects of the algebras $C^n(\cam, H)$ of $n$ times higher weakly $D-$differentiable elements inside $\cam,$ but unlike the case in noncommutative geometry we will not assume that $C^n(\cam, D)$ is dense in $\cam$ in any ordinary topology.
The prototype of a von Neumann algebra, or rather the commutative example which may give inspiration for general results  on von Neumann algebras is the algebra of measurable essentially bounded function on the unit circle, $L^{\infty}(\bt, d\theta ),$ and in this setting $C^n(\cam, D)$ is nothing but the $n$ times weakly $D:= -i\frac{d}{d\theta}-$differentiable functions, so we find  here that for $n \geq 1$ we have   $C^n(L^{\infty}(\bt, d\theta ), D) = C^n(C(\bt), D),$ and we may wonder if the von Neumann algebra property plays a role at all ? We have no answer, but this might be because our understanding of the relations between noncommutative and commutative geometry is still quite limited. Below we will describe the property called reflexivity of an algebra of bounded operators, but for the moment just say, that a von Neumann algebra $\cam$  is reflexive and that property is partly inherited by $C^n(\cam, D),$ in the sense  that this algebra has a representation as a reflexive algebra of bounded operators on a Hilbert space.  We will describe the reflexivity  property in details below, but right now we will like to mention that reflexivity is a very strong property for an algebra of operators to have. This follows from von Neumann's bicommutant theorem which shows that if an  algebra of bounded operators on $H$  is self-adjoint and reflexive then it is a von Neumann algebra. The algebras we will study are not self-adjoint, but sub-algebras of the upper triangular matrices in $M_{n+1}(B(H)),$ so von Neumann's theorem does not apply directly in our situation. 

We will remind you of the definition of reflexivity as it was defined by Halmos and described in the book \cite{RR}. 
\begin{definition} Let $H$ be a Hilbert space.
\begin{itemize}
\item[(i)]
Let $\cas$ be a set of bounded operators on a Hilbert space $H$ then $\mathrm{Lat}(\cas)$ is  the lattice of closed subspaces of $H$ which are left invariant by each of the operators in $\cas.$
\item[(ii)] Let $\cg$ denote a collection of closed subspaces of $H$ then $\mathrm{Alg}(\cg)$ is the algebra of bounded operators on $H$ which leave each of the subspaces in $\cg$ invariant. 
\item[(iii)] A subalgebra $\car$ of $B(H)$ is said to be reflexive if \newline $\car = \mathrm{Alg}( \mathrm{Lat}(\car)).$
\end{itemize}
\end{definition}
One of the strong properties of a reflexive algebra of bounded operators on a Hilbert space $K$ is that it is an ultraweakly closed subspace of $B(K)$ and then it becomes a dual space since all the ultraweakly continuous functionals on $B(K)$ form  the predual of $B(K).$ 
Then the reflexive algebra has a predual which is a quotient of the predual of $B(K).$ In the set-up for the classical commutative differential geometry such kinds of dualities are well known and widely used. The reflexivity is actually stronger than this duality property, but so far we have not been able to single out a property which solely depends on the reflexivity of a certain representation of 
the algebra $C^n(\cam, D).$ We will now formulate the result:

\begin{theorem}
Let $\cam $ be a von Neumann algebra on a Hilbert space $H,$  $D$ a self-adjoint operator on $H$ and $n$ a non-negative integer. There exists a unital injective algebraic homomorphism $\Phi_n : C^n(\cam, D) \to B(H\otimes \bc^{n+1}) $ such that the image $\car_n(\cam,D) :=  \Phi_n(C^n(\cam, D))$ is a reflexive algebra on $H\otimes\bc^{n+1}.$
 
For any $x$ in $C^n(\cam, D): \frac{1}{n+1}\|x\|_n \leq \|\Phi_n(x)\| \leq \|x\|_n.$ 
  \end{theorem}
\begin{proof}
For $n =0$ we have $C^0(\cam, D) = \cam, $ and then $C^0(\cam, D)$ is a reflexive subalgebra of $B(H),$ so we define $\Phi_0 := \mathrm{id}\big|C^0(\cam, D),$ and $\car_0 := \cam.$  For $n > 0$ we will construct a representation $\Phi_n$ of $C^n(\cam, D)$ into the upper triangular matrices with constant diagonals inside the $(n+1)\times (n+1)$ matrices over $B(H)$ such that for an $x$ in $C^n(\cam,D)$ the representation is given by 

$$ \Phi_n(x) := \begin{pmatrix} x & \d_w(x) & \frac{1}{2}\d_w^2(x)& .&.& . & \frac{1}{n!}\d_w^n(x) \\ 
0 & x & \d_w(x) &. &.&. & \frac{1}{(n-1)!•}\d_w^{n-1}(x) \\
.& . & .& .&. &. & . \\
.& . &.& .&. &. & . \\
0& . &.& .&. &\d_w(x) & \frac{1}{2}\d_w^2(x) \\
0 & . &.& . & .& x& \d_w(x) \\
0 & . &.& . &. &0  & x \end{pmatrix}$$

\noindent
and the element in the $j$'th upper diagonal is $\frac{1}{j!}\d_w^j(x).$

We define $\car_n := \Phi_n(C^n(\cam,D)).$
If $D$ is bounded then $\d_w(x) = [iD,x]$ and it is well known that the mapping $\Phi_n$ is a homomorphism and $\car_n$ is an algebra. But now $\d_w(x)$ is the closure of the commutator $[iD,x] $ so elementary algebra does not apply right away. The short proof of the homomorphism property  is then that the results of Theorem \ref{mdn} show that the algebraic arguments are still valid when restricted to take place on the domain dom$(D^n)$ only. We will like to show this with some more details because these arguments will be needed, when we want to show the reflexivity of  $\car_n.$ To set the stage we define $B_n$ as the matrix in $M_{n+1}( \bc) $ with ones in the first upper diagonal and zeros elsewhere. 
$$B_n:= \begin{pmatrix}
0&1&0 &. & 0\\
0&0&1& . &0 \\
. &.& .& .& . \\
 0 & . & .& 0& 1 \\
 0 & . & .& 0& 0 \\
\end{pmatrix}$$

Then $B_n$ is nilpotent and satisfies $B_n^{(n+1)} = 0,$ which will be very useful in the computations to come. First we can describe  $\Phi_n(x)$ inside the tensor product $B(H) \otimes M_{n+1}(\bc) $    as 
$$\Phi_n(x) = x \otimes I + \sum_{j=1}^n \frac{1}{j!} \d_w^j(x) \otimes B_n^j, $$ and we see from Theorem \ref{mdn} that all the elements in the sum are defined as elementary operator theoretical products or sums of such products on the space  dom$(D^n)\otimes \bc^{(n+1)}. $ We will then define $\cd_n = \mathrm{dom}(D^n) \otimes \bc^{(n+1)},$ and the coming computations will all take place on this dense subspace of $H \otimes \bc^{(n+1)}.$  
We will work with matrices of unbounded operators and the first, denoted $S_n$  is defined as 
 $$ S_n :=  iD \otimes B_n, \quad \mathrm{dom}(S_n) =  H\oplus \mathrm{dom}(D)  \oplus \dots \oplus \mathrm{dom}(D).$$

In order to be able to talk on specific matrix elements we suppose that $\bc^{(n+1)}$ is equipped with its canonical basis, and that the basis elements $e_j$ are numbered by $0 \leq j \leq n.$ Then the matrix elements are indexed by $\{ij\}$ with $i,j \in\{0, \dots, n \}$ too.    
   
For $x$ in $C^n(\cam, D)$ the Theorem \ref{mdn} shows that for any $j$ in $\{1, \dots, n\} $ we have $x$dom$(D^j) \subseteq $dom$(D^j)$ so for any set of natural numbers \newline  $j_1, \dots, j_k$ with $j_1 + \dots +j_k \leq n,$  any set of operators $x_0, x_1, \dots, x_k$ in $C^n(\cam, D)$  and any vector $\xi $ in dom$(D^n),$ the vector $\xi$  will be in the domain of definition for $ x_0D^{j_1}x_1 \dots D^{j_k}x_k.$
We will lift this product to the matrices, and in order to do so we  introduce the canonical amplification $\iota(x) $ of $B(H) $ into $M_{(n+1)} (B(H))$ by $\iota(x) := x \otimes I.$ Then for $x_0, x_1, \dots , x_k$ in $C^n(\cam, D)$ we can define a product of operators which always will be defined on  $\cd_n$ by the following convention.

 \begin{align*} &   \iota(x_0) S_n^{j_1}   \iota(x_1) \dots S_n^{j_k}\iota(x_k) \big| \cd_n \\  
 := &\begin{cases} 0 \big|\cd_n \qquad \text{ if } j_1 + \dots + j_k > n \\
\big(\big( x_0(iD)^{j_1} x_1 \dots (iD)^{j_k}x_k \big) \otimes B_n^{(j_1 + \dots + j_k)}\big)\big| \cd_n   \text{ if } j_1 + \dots + j_k \leq n. \end{cases}
\end{align*}

This means that $S_n$ and $\iota(C^n(\cam,D))$ generate an algebra with this special product. The product is a bit more complicated than just the product of the restrictions to $\cd^n$ of each of the factors.  This is because the operator $D$ does not map dom$(D^n)$ into dom$(D^n),$ so 
$S_n$ does not map $\cd_n$ into $\cd_n$  but anyway all the products mentioned make sense by first making the standard operator product and then restricting the outcome to $\cd_n.$ We can then define $\ct_n$ as the algebra of matrices 
defined on $\cd_n$ with this product and  generated by $S_n$ and $\iota(C^n(\cam,D))$ The point of this is that we may now use standard algebra on this associative unital algebra and we define elements $T_n$  and its inverse $T_n^{-1}$ by exponentiating  $S_n.$ The nil-potency of $S_n$ gives us the following formulas inside  this algebra:
\begin{align*}T_n &:= \exp(S_n) = I\otimes I + \sum_{j=1}^n \frac{1}{j!} S_n^j \\
  T_n^{-1} &:= \exp(-S_n) = I\otimes I + \sum_{j=1}^n \frac{1}{j!} (-S_n)^j. 
  \end{align*} 
  
In order to relate  $S_n$ and $T_n$ to the unital algebra $\car_n$ we remind you that in any unital associative algebra $\cc$  with a nilpotent element $s$ we may study the derivation ad$(s) $ on $\cc$ given by ad$(s)(x) := [s,x]$ and we have that $\exp\big(\mathrm{ad}(s)\big)(x) = \exp(s) x \exp(-s).$   
In our setting we then get that for any $x$ in $C^n(\cam,D)$ we have 
$$
\mathrm{ad}(S_n)^j (\iota(x)) \big| \cd_n = \begin{cases} \d_w^j(x) \otimes B_n^j \big|\cd_n \text{ if } j \leq n \\
0\big|\cd_n \text{ if } j > n,\end{cases}  $$ so the equalities above yield the following identities in the algebra $\ct_n.$ 

\begin{align} \label{Phi}   
T_n\iota(x)T_n^{-1} &= \exp(S_n) \iota(x) \exp(-S_n) \\ \notag  
&= \exp(\mathrm{ad}(S_n)) (\iota(x)) \\ \notag  
&= \big(\iota(x) + \sum_{j=1}^n \frac{1}{j!} \d_w^j(x)\otimes B_n^j\big)\big|\cd_n \\\notag  
&= \Phi_n(x) \big|\cd_n. 
\end{align}
  
We can now find a family $\cl_n$ of   closed $\car_n$ invariant subspaces of $H \otimes \bc^{(n+1)} $ such that we will have $\car_n = \mathrm{Alg}(\cl_n),$ and in this way the reflexivity of $\car_n$ will be established. The proof is made by induction  and for the case of $n=0$ the family $\cl_0 $ is just the set Lat$(\cam)$ of closed subspaces which are invariant under any element in $\cam,$ and  and it follows from von Neumann's bicommutant theorem that $C^0(\cam, D ) = \cam = \mathrm{Alg}(\cl_0).$ 
We now assume that $n \geq 1 $ and we define a subset $\cl_n$  of Lat$(\car_n)$ which is so big that  Alg$(\cl_n)  = \car_n.$
Since we are forming an induction argument it is convenient to think of the Hilbert spaces $H\otimes \bc^{(n+1)},$ as a nested family of closed subspaces of $\ell^2(\bn_0, H)$ in the following way,
\begin{align*}  H_0 &:= H \otimes e_0 \\
H_n &:= H\otimes e_0 \oplus \dots \oplus H \otimes e_n,\end{align*}
\noindent
and we will let $K = \ell^2(\bn_0, H)$ and let $E_n$ denote the orthogonal projection of $K$  onto $H_n.$
For each $n $ we will also identify $H_n$ with the abstract tensor product $H\otimes \bc^{(n+1)}$ in the way that an expression $\xi\otimes e_j$ which appears in both spaces are identified.  In this way $\car_n$ may be identified with some upper triangular matrices whose entries are 0 whenever any index is bigger than $n,$ or described as a subspace of the  bounded operators  on $K$ which satisfies $X= E_nXE_n.$ We then see that for $0 \leq j \leq n $ the subspace $H_j$ is invariant for $\car_n,  $ and we also note that for any closed subspace $F$ in Lat$(\cam)$ the subspace $F\otimes e_0$ is invariant for the algebra $\car_n$ too. We will point out 2 more, but closely related examples of  closed subspaces of $H_n$ which are invariant for  $\car_n,$ and we will denote these spaces  $P_n$ and $Q_n.$  First it is practical to redefine $\iota(x)$ to act on $H \otimes \ell^2(\bn_0)$ by   $\iota(x) := x \otimes I_{\ell^2(\bn_0} $ and also redefine $B_n$ as the canonical image of $B_n$ in $B(K) $ under the embedding of $\car_n$ into $E_nB(K)E_n.$  For a natural number $n $ we define a  subspace $P_n$ of $H_n$ by 
\begin{align} 
P_n :&= \{ T_n (\xi \otimes e_n) \, : \, \xi \in \mathrm{dom}(D^n) \} \label{Pn} \\ &=  \{\xi \otimes e_n + \sum_{j=1}^n \frac{1}{j!} (iD)^j \xi  \otimes e_{n-j}\, : \, \xi \in \mathrm{dom}(D^n) \,\} , \notag
\end{align} so this space  is just  the graph of a certain    operator $V_n$ from dom$(D^n)\otimes e_n $ to $H_n.$   
By the closedness of all the powers $(iD)^j$ we see that  $V_n$ is a closed operator, so   $P_n$ is a closed subspace of $H_n$ and 
by the relation (\ref{Phi})
we get that for any $x$ in $C^n(\cam,D)$ and any  $\xi $ in dom$(D^n)$ we  have

\begin{align}
x \xi &\in \mathrm{dom}(D^n) \text{ since } x \in C^n(\cam, D), \\ \notag
 \Phi_n(x) T_n \xi \otimes e_n &=   T_n \iota(x) \xi \otimes e_n = T_n (x \xi) \otimes e_n, \label{cyc}
 \end{align} 
so $P_n$ is an invariant subspace for $\car_n$ acting on $H_n.$  If $j <   n$  then for any  $x$ in $C^n(\cam,D)  \subseteq C^j(\cam, D)$  we see by the  construction of $\Phi_j(x)$ and $\Phi_n(x)  $  that $\Phi_n(x) \big| H_j  = \Phi_j(x)$ Hence for any $j < n $ we also have that $P_j$ is a closed  invariant subspace for $\car_n.$ 

To construct the last invariant subspace we remind you that if we define $\widetilde{D} := D+I$ then $\widetilde{D} $ is also a self-adjoint operator, and since $\exp(it\widetilde{D}) = e^{it}\exp(itD)$ the corresponding  automorphism groups $\widetilde{\a}_t$ and $\a_t$ are identical so for any $n$ in $\bn_0$ we have $C^n(\cam, \widetilde{D}) = C^n(\cam, D) $ and $\widetilde{\d}_w^n = \d_w^n.$ In particular $\widetilde{\car}_n = \car_n$ so a closed subspace of $H_n,$  which is invariant for $\widetilde{\car}_n$ is also invariant for $\car_n.$ We may then repeat the construction made for $P_n$ but now based on $D+I$ to obtain the invariant subspace $Q_n$ which is obtained via the equations below

\begin{align}   
\widetilde{T_n} :&= I\otimes E_n + \sum_{j=1}^n \frac{1}{j!}(i(D+I))^j \otimes B_n^j\\ \label{Qn} Q_n :&= \{\widetilde{T}_n (\xi \otimes e_n) : \xi \in \mathrm{dom}(D^n) \} \\ &=\{\xi \otimes e_n + \sum_{j=1}^n \frac{1}{j!} (i(D+I))^j \xi  \otimes e_{n-j}\, : \, \xi \in \mathrm{dom}((D+I)^n) \,\}   \notag  
\end{align}  
We need   to remark that dom$((D+I)^n)$ equals dom$(D^n),$ and that statement  follows from the binomial formula and the fact that for $j\leq n $ we have dom$(D^n) \subseteq \mathrm{dom}(D^j).$

 We can then define the collection $\cl_n$ of $\car_n$ invariant subspaces of $H_n$ by 
 $$ \cl_n := \cl_0 \cup \{H_j \, : \, 0 \leq j \leq n\, \} \cup \{P_j\, : 1 \leq j \leq n \} \cup \{Q_j\, : 1 \leq j \leq n \}.$$
 
 It is clear that the algebra  Alg$(\cl_n)$ will contain the unit $I$ of $B(K),$  which can never be an element $\car_n,$ whose matrices all have zero entries outside the upper $(n+1)\times (n+1) $ corner, but we will prove by induction that $$\car_n = \{X \in \mathrm{Alg}( \cl_n) \, : \, E_nXE_n = X. \}$$
    The case $n=0$ is already established, so let us assume that $n> 0$ and  the statement is true for  $n-1, $ and let  $X$ be an  operator in Alg$(\cl_n)$ such that $E_nXE_n = X.$  Then $H_{(n-1)}$  is an invariant subspace for $X$ so $XE_{(n-1)} = E_{(n-1)} X E_{(n-1)} $ and  we find immediately that $XE_{(n-1)} $ also leaves all the subspaces in $\cl_{(n-1)} $ invariant and the induction hypothesis tells that there exists an operator $x$ in $C^{(n-1)}(\cam,D) $ such that $XE_{(n-1)} = \Phi_{(n-1)}(x).$ Unfortunately we do not know that the operator $x$ is in $C^n(\cam, D) $ too, but we will show it now and then prove that $X = \Phi_n(x).$  We know that $P_n$ and $Q_n$ are  invariant subspaces for $X$ and from the  equations ( \ref{Pn}) and  ( \ref{Qn})  we  have  descriptions of $P_n$ and $Q_n$ which will become useful. 
Hence let $\xi $ be in dom$(D^n),$ $T_n \xi \otimes e_n$ and $\widetilde{T}_n \xi \otimes e_n$ be the corresponding vectors in $P_n$  and  $Q_n$ respectively.       The invariance of $P_n $ under $X$ has as its first consequence that for the operator entry $x_{nn} $ of $X$ we get $x_{nn} \xi $ is in dom$(D^n). $ If we look at the $(n-1)$'st coordinate of the vector $XT_n \xi  \otimes e_n$ the invariance of $P_n$ under $X$ implies the equation
\begin{equation} \label{xnn}
x(iD)\xi + x_{(n-1)n}\xi = (iD) x_{nn} \xi. 
\end{equation}  By analogy we get a similar equation based on the invariance of $Q_n$ under $X$ so we get 
\begin{equation}
x(i(D+I))\xi + x_{(n-1)n} \xi = (i(D+I)) x_{nn} \xi. 
\end{equation} 
By subtraction of those equations we get \begin{equation} \label{domDn}
\forall \xi \in \mathrm{dom}(D^n): \quad  x \xi =  x_{nn} \xi, 
\end{equation} so since both operators are bounded we have $x_{nn} = x.$  The equation
 (\ref{xnn}) may then be applied to show that $x_{(n-1)n} = \d_w(x)$ and it is possible to continue along this line to show that $X = \Phi_n(x),$ but we will instead address the first element, of the vector $XT_n \xi\otimes e_n,$ since it seems to be easier to write down the details in this case. Let us return to the general setting we studied just in front of the equation (\ref{Phi}) where we have an associative unital algebra $B$  and an element $s$ in $B.$ We will then define operators $L$ and $R$ on $B$ by left and right multiplications by $s,$ so $Lb := sb $ and $Rb:= bs.$ then by the binomial formula we get, 
 since $L$ and $R$ commute that  for any $b$ in $B$ 
 \begin{align}  
 \sum_{j=0}^n \frac{1}{(n-j)!}\frac{1}{j!}\mathrm{ad}(s)^j(b) s^{(n-j)} &= \sum_{j=0}^n \frac{1}{(n-j)!}\frac{1}{j!}(L-R)^jR^{(n-j)}b \label{row0}\\ & =  \frac{1}{n!}L^nb = \frac{1}{n!}s^nb. \notag
 \end{align} 
We will recall the commutator mapping $d$ which we defined in Definition 3.1 as $d(x) := [iD,x]. $ We know from above that $x $ is in $C^{(n-1)}(\cam, D)$ so by Theorem \ref{mdn} for any $j$ in $\{1, \dots , n-1\} $ we have $x$dom$(D^j) \subseteq $ dom$(D^j),  $ from equation ( \ref{domDn}) we have $x$dom$(D^n) \subseteq $ dom$(D^n),  $  so all the expressions $d^j(x) $ are defined on dom$(D^n).$ The algebraic identity (\ref{row0}) then applies and we get
 \begin{equation} \label{did} 
 \sum_{j=0}^n \frac{1}{(n-j)!}\frac{1}{j!}d^j(x) (iD)^{(n-j)}\big|\mathrm{dom}(D^n) = \frac{1}{n!}(iD)^nx\big|\mathrm{dom}(D^n).
 \end{equation}
Since $x$ is in $C^{(n-1)}(\cam,D)$ we have $$d^j(x)\big|\mathrm{dom}(D^n) = \d_w^j(x)\big|\mathrm{dom}(D^n), \text{ for } 0 \leq j \leq n-1.$$

On the other hand the invariance  of $P_n$ shows that 
\begin{align*}  
 &\big((\sum_{j=0}^{(n-1)}  \frac{1}{(n-j)!}\frac{1}{j!}\d_w^j(x) (iD)^{(n-j)}) \, + \, x_{0n}\big)\big|\mathrm{dom}(D^n)\\ = &\frac{1}{n!}(iD)^nx\big|\mathrm{dom}(D^n).
 \end{align*}
 
 By elementary algebra we then get that $$\frac{1}{n!}d^n(x)\big|\mathrm{dom}(D^n) = x_{0n}\big|\mathrm{dom}(D^n),$$ 
 so by Theorem \ref{mdn} we find that $x$ is in $C^n(\cam, D)$ and that $\frac{1}{n!}\d_w^n(x) = x_{0n},$ as expected. Recall that by (\ref{cyc}) $P_n $ is invariant under the elements in $\car_n, $ and with this in mind we get that $P_n$ must be invariant under $Y := (X - \Phi_n(x)),$ which is a column matrix  such that $y_{ij} = 0$ 
 whenever $j \neq n, $ and also satisfies $y_{nn} = 0,$ which is crucial for the next argument.  Given any vector $\xi $ in dom$(D^n)$ with
  corresponding vector $T_n(\xi \otimes e_n) $ in $P_n$ we see that $y_{nn}=0$ implies that the $n$'th coordinate of $YT_n(\xi \otimes e_n)$ is equal to $0,$ but then $YT_n(\xi \otimes e_n) = 0,$ since the space $P_n$ may be thought of as the graph of an operator defined on the last coordinate, which here vanishes. On the other hand, for the given $\xi$ in dom$(D^n)$  we get $0=YT_n(\xi \otimes e_n ) = \sum_{i=0}^{n} (y_{in}\xi) \otimes e_i.$ Hence for $0 \leq i \leq n $ we get $y_{in} = 0$  and then $X = \Phi_n(x)  $ for an $x$ in $C^n(\cam, D) $ and the reflexivity of $\car_n$ is proven. 
  \end{proof}

\bibliographystyle{amsplain}

\end{document}